\documentclass{amsart}

\usepackage{amsmath,amssymb,amsfonts,amsthm,mathtools,mathrsfs}
\usepackage[mathscr]{eucal}
\usepackage{cases,empheq,relsize}
\usepackage{color,graphicx,epstopdf,tikz}
\usetikzlibrary{arrows,automata,backgrounds,patterns,calc}
\usepackage{array,tabularx,varwidth,enumerate,verbatim,xparse}
\usepackage[normalem]{ulem}
\usepackage[linkbordercolor=blue]{hyperref}
\hypersetup{
  colorlinks=false,
  linkcolor=blue,
  citecolor=red,
  filecolor=magenta,
  urlcolor=cyan
}

\numberwithin{equation}{section}
\theoremstyle{plain}
\newtheorem{thm}{Theorem}[section]
\newtheorem{cor}[thm]{Corollary}
\newtheorem{lem}[thm]{Lemma}
\newtheorem{prop}[thm]{Proposition}
\theoremstyle{remark}
\newtheorem{rmk}[thm]{Remark}

\newcommand{\R}{\mathbb R}
\newcommand{\C}{\mathbb C}
\newcommand{\Z}{\mathbb Z}
\newcommand{\Sf}{S_m^\gamma f}

\author[Jeong]{Eunhee Jeong}
\address[Eunhee Jeong]{%
  Department of Mathematics Education, and Institute of Pure and Applied Mathematics, Jeonbuk National University, Jeonju 54896, Republic of Korea}
\email{eunhee@jbnu.ac.kr}

\author[Ko]{Hyerim Ko}
\address[Hyerim Ko]{%
  Department of Mathematics, and Institute of Pure and Applied Mathematics, Jeonbuk National University, Jeonju, 54896, Republic of Korea}
\email{kohr@jbnu.ac.kr}

\begin{document}

\title[Orthonormal maximal estimates along H\"older curves]
{Orthonormal maximal estimates for fractional Schr\"odinger operators along H\"older curves}

\begin{abstract}
We establish orthonormal maximal estimates for one-dimensional fractional Schr\"odinger operators along curves that are H\"older continuous in time and bi-Lipschitz in space.
For both the convex range $m>1$ and the concave range
$0<m<1$, we obtain sharp Sobolev regularity up to endpoint in terms of the order $m$. As an application, we obtain almost everywhere convergence of the associated densities to their initial values along the curves. The proof combines Schatten duality with kernel estimates.  
\end{abstract}

\maketitle

\section{Introduction}

\subsection{Classical pointwise convergence}
Let $d\ge1$ and $m>0$. We consider the fractional Schr\"odinger equation
\begin{align}\label{frac sch}
i\partial_t u  + (-\Delta)^{m/2}u=0, \quad u(x,0)=f(x),
\end{align}
with initial data $f\in H^s(\R^d)$. Here, $H^s(\mathbb R^d)$ is the inhomogeneous Sobolev space equipped with the norm $\|f\|_{H^s}=\|\langle D\rangle^sf\|_{L^2}$, where $\langle D\rangle=(1-\Delta)^{1/2}$. The solution of \eqref{frac sch} is formally given by 
\[
e^{it(-\Delta)^{m/2}}f(x)=(2\pi)^{-d}\int e^{ix\cdot \xi+it|\xi|^m} \widehat f(\xi)\,d\xi
\]
where $\widehat f(\xi)=\int_{\R^d} e^{-ix\cdot\xi}f(x)\,dx$. The pointwise convergence problem is to determine the smallest regularity $s$ for which
\begin{align}\label{pt conv}
\lim_{t\rightarrow0} e^{it(-\Delta)^{m/2}}f(x)=f(x), \quad \text{a.e.} \ x\in \R^d
\end{align} 
holds for every $f\in H^s(\mathbb R^d)$.

When $m=2$, \eqref{frac sch} is the Schr\"odinger equation. 
In one dimension, Carleson \cite{Carleson} proved \eqref{pt conv} for $s\geq1/4$, and Dahlberg--Kenig \cite{DK82} showed that this condition is sharp. In higher dimensions, Du--Guth--Li \cite{DGL17} for $d=2$ and Du--Zhang \cite{DZ19} for $d\ge3$ established \eqref{pt conv} for $s>\frac{d}{2(d+1)}$. This range is sharp up to endpoint by Bourgain's counterexample \cite{Bourgain16}. See also the references therein for earlier developments.

When $m>1$, Sj\"olin \cite{Sjolin87} showed that, in one dimension, \eqref{pt conv} holds for every $f\in H^s(\mathbb R)$ if and only if $s\ge1/4$ (see also \cite{KPV91}). 
In higher dimensions, he obtained convergence for $s\ge1/2$ when $d=2$ and for $s>1/2$ when $d\ge3$. Subsequently, further progress has been made by several authors \cite{MYZ, CK22} for $d=2$.

When $0<m<1$, Cowling \cite{Cowling} proved \eqref{pt conv} for $s>m/2$ in every dimension. In one dimension, Walther \cite{Walther} showed that the associated $L^2$ maximal estimate holds for $s>m/4$ and fails for $s<m/4$. When $m=1$, which corresponds to the wave equation, \eqref{pt conv} holds if and only if $s>1/2$ in any dimension. See \cite{Cowling,Walther2}.

\subsection{Convergence along curves}
Our main focus is the one-dimensional problem of pointwise convergence along curves. Let $B=[-1,1]$ and $I=[0,1]$, and let $\gamma$ be a continuous map satisfying
\begin{align}\label{curve define}
\gamma:B\times I\rightarrow \R, \qquad \gamma(x,0)=x.
\end{align} 
The fractional Schr\"odinger operator along $\gamma$ is given by 
\begin{equation*}
\Sf(x,t)
=\frac{1}{2\pi}\int_\mathbb R e^{i\gamma(x,t)\xi+it|\xi|^m}\widehat f(\xi)\,d\xi.
\end{equation*}
We are interested in determining the smallest regularity $s$ so that
\begin{align}\label{pt along curve}
\lim_{t\rightarrow0} \Sf(x,t)=f(x), \quad \text{a.e.} \ x\in \R
\end{align}
for every $f \in H^s(\R)$. The choice $\gamma(x,t)=x$ recovers the convergence along the vertical line \eqref{pt conv}.

For the Schr\"odinger equation $(m=2)$, Sj\"ogren--Sj\"olin \cite{SjSj} showed that convergence within a non-tangential region can fail for $s\le d/2$. Johansson \cite{Johansson} extended this counterexample to a class of dispersive symbols that includes $|\xi|^m$ for every $m>1$. Along a curve $\gamma$ satisfying \eqref{curve define}, lower regularity may suffice. Motivated by the quantum harmonic oscillator, Lee--Rogers \cite{LeeRogers} showed that vertical maximal estimates imply convergence along $C^1$ curves with the same regularity, up to endpoints. 
More generally, curves that are H\"older continuous in $t$ and bi-Lipschitz in $x$, allowing both tangential and non-tangential approaches, were studied by Cho--Lee--Vargas \cite{CLV}. 

To be precise, let $0<\kappa\le1$. We denote by $\Gamma(\kappa)$ the class of continuous functions $\gamma$ satisfying \eqref{curve define} for which there exist constants $A_\gamma>0$ and $B_\gamma>1$ such that the $\kappa$-H\"older
condition in the time variable
\begin{equation}\label{holder-k}
 |\gamma(x,t)-\gamma(x,t')|
 \leq A_\gamma |t-t'|^\kappa,
\end{equation}
and the bi-Lipschitz condition in the spatial variable
\begin{equation}\label{lip-x}
 B_\gamma^{-1}|x-y|
 \leq
 |\gamma(x,t)-\gamma(y,t)|
 \leq
 B_\gamma|x-y|
\end{equation}
hold for all $x,y\in B$ and $t,t'\in I$. Since $\gamma(x,0)=x$, the H\"older condition \eqref{holder-k} implies $|\gamma(x,t)-x|\le A_\gamma t^\kappa$. Hence, if $\kappa=1$, the point
$(\gamma(x,t),t)$ remains in the cone $\{(y,t):|y-x|\le A_\gamma t\}$ and converges to $(x,0)$ \textit{non-tangentially} with respect to the line $\{t=0\}$, whereas for $0<\kappa<1$ the class $\Gamma(\kappa)$ also contains curves along which the convergence is \textit{tangential}. Typical examples are the curves
\begin{equation*}
 \gamma_{\theta}^{\rho}(x,t):=x-\theta t^{\rho},\qquad \rho,\theta>0,
\end{equation*}
which belong to $\Gamma(\kappa_\rho)$ with $\kappa_\rho:=\min\{\rho,1\}$. These are \textit{tangential} for $0<\rho<1$ and \textit{non-tangential} for $\rho\ge1$.  Note that $\Gamma(1)\subseteq\Gamma(\kappa)$ for every $0<\kappa\leq1$, and the vertical line $\gamma(x,t)=x$ belongs to all these classes. The restriction $\kappa\le1$ is natural, since any function satisfying the H\"older condition \eqref{holder-k} with $\kappa>1$ must be constant in the time variable. Hence, the condition $\gamma(x,0)=x$ forces $\gamma(x,t)=x$.

The following theorem summarizes the sharp $L^2$ maximal estimates for curves in $\Gamma(\kappa)$ when $m>0$ and $m\neq1$.

\begin{thm}[{\cite{CLV,CS21,YZ}}]\label{thm:single}
Let $d=1$, $m>0$, $m\neq1$, $0<\kappa\le 1$, and $\gamma \in \Gamma(\kappa)$. Then the maximal estimate
\begin{equation}\label{maxL2}
\big\|\sup_{t\in I} |\Sf(\cdot,t)|\big\|_{L^2(B)} \lesssim \|f\|_{H^s(\mathbb R)}
\end{equation}
holds for all $f\in H^s(\R)$ in the following cases:
\begin{enumerate}[$(i)$]
	\item $s>\max\big\{\frac14,\frac{1-m\kappa}2\big\}$ if $m>1$,
	\item $s>\max\big\{\frac12-\frac m4, \frac{1-m\kappa}2\big\}$ if $0<m<1$.
\end{enumerate}
\end{thm}

The critical regularity in Theorem \ref{thm:single} is sharp up to endpoint. For the model curve $\gamma(x,t)=x-t^\kappa$, the estimate \eqref{maxL2} fails if $s<\max\{\frac14,\frac{1-m\kappa}{2}\}$ when $m>1$, and if $s<\max\{\frac12-\frac m4,\frac{1-m\kappa}{2}\}$ when $0<m<1$. See Proposition \ref{prop:nece}.

The case $(i)$ is due to Cho--Lee--Vargas \cite{CLV} for $m=2$ and to Cho--Shiraki \cite{CS21} for general $m>1$, and the case $(ii)$ is due to Yuan--Zhao \cite{YZ}. A standard density argument shows that \eqref{maxL2} implies  \eqref{pt along curve}. See, for example, \cite{CLV,CS21}.

We note that the two cases behave differently in comparison with the vertical line $\gamma(x,t)=x$. For non-tangential curves $\gamma \in \Gamma(1)$, Theorem \ref{thm:single} gives the critical regularity $s>1/4$ when $m>1$ and $s>1/2-m/4$ when $0<m<1$. The former coincides with the vertical case, whereas the latter is strictly larger than the vertical case $s>m/4$ obtained by Walther \cite{Walther}. This loss is already necessary for the non-tangential curves $\gamma_\theta^\rho$ with $\rho\ge1$ and $\theta>0$. Cho--Shiraki \cite{CS22} showed that $s>\frac12-\frac m4$ is sharp for the associated $L^2$ maximal estimates, up to endpoint.

Several refinements are also known. Cho--Lee \cite{CL14} obtained bounds for the dimension of the divergence sets when $m=2$. The results of \cite{CS21,YZ,CS22} include maximal estimates with respect to $\alpha$-dimensional measures. 
Li--Wang \cite{LW} established sharp $L^p_x$ maximal estimates along H\"older curves for $m=2$, which are recovered for $p\ge 2$ by the single function case of Theorem \ref{thm:max}.

In higher dimensions, several results are known for $m=2$. For curves of the form $x+\eta(t)$, where $\eta(0)=0$ and $\eta$ is $\kappa$-H\"older continuous with $1/2\le \kappa<1$, Li--Wang \cite{LW} obtained convergence for $s>3/8$ when $d=2$. Cao--Miao \cite{CM} obtained sharp regularity $s>d/(2(d+1))$ for $d\ge2$ and a class of curves of the form $x+t^\kappa \eta(t)$ satisfying additional smoothness and derivative bounds. Minguill\'on \cite{Ming} proved convergence for $s>\frac{d}{2(d+1)}$ for an arbitrary $\kappa$-H\"older curve $\eta$ with $1/2\le \kappa \le 1$. This is sharp up to endpoints and coincides with that for the vertical line. For the model case $\eta(t)=(t^{\kappa_1},\dots,t^{\kappa_d})$, where $\kappa_j>0$ and $\kappa=\min_j\kappa_j<1/2$, Minguill\'on--Soria--Vargas \cite{MSV} determined the critical regularity $\max\{\frac{1-2\kappa}{2},\frac{d}{2(d+1)}\}$, which is sharp up to the endpoint.

\subsection{Orthonormal maximal estimates}
We extend the pointwise convergence problem along curves \eqref{pt along curve} to orthonormal systems by establishing maximal estimates of the form
\[
\Big\| \sum_j \mu_j| \Sf_j|^2 \Big\|_{L_x^{\frac q2}(B;L_t^\infty(I))} \lesssim \|\mu\|_{\ell^\beta},
\]
where $2\le q \le \infty$, $1\le \beta \le \infty$, $(f_j)$ is orthonormal in $H^s(\R)$, and $\mu=(\mu_j)\subset \C$. Here, orthonormality in $H^s(\R)$ means that $(\langle D\rangle^s f_j)$ is orthonormal in $L^2(\R)$. 

The case $\beta=1$ follows from the corresponding single-function estimates by the triangle inequality. Since the estimate becomes stronger as $\beta$ increases, the goal is to determine the optimal summability exponent by exploiting orthogonality.

Classical inequalities for orthonormal systems include the Lieb--Thirring inequality \cite{LT} and Lieb's Hardy--Littlewood--Sobolev inequality \cite{Lieb83} (see \cite{Frank23} for a survey). Bez--Rogers--Toyoshima \cite{BRT26} recently extended Lieb's Hardy--Littlewood--Sobolev inequality for orthonormal functions to fractal measures.

In the dispersive setting, orthonormal Strichartz estimates were first established by Frank--Lewin--Lieb--Seiringer \cite{FLLS} and further developed by Frank--Sabin \cite{FS_AJM}. More recently, Hadama--Toyoshima \cite{HT26} proved a conjectured endpoint estimate with a Lorentz sequence norm in dimension $d\ge2$. Bez--Hong--Lee--Nakamura--Sawano \cite{BHLNS} considered initial data in homogeneous Sobolev spaces, and Bez--Lee--Nakamura \cite{BLN_Forum} obtained analogous estimates for other dispersive equations, including the fractional Schr\"odinger equation.

Maximal estimates and pointwise convergence for orthonormal systems were first studied by Bez--Lee--Nakamura \cite{BLN_Selecta} for the Schr\"odinger equation in one dimension, and were extended by Bez--Kinoshita--Shiraki \cite{BKS} to fractional Schr\"odinger equations and to higher dimensions. In one dimension, the proof of \cite[Theorem~2.2]{BKS} yields the following: for $1\le\beta<\infty$ and $0\le s<1/2$, the estimate
\begin{align}\label{BKS-max}
\Big\|\sum_j\mu_j
\big|e^{it(-\Delta)^{m/2}}f_j\big|^2\Big\|_{L_x^\beta (B;L_t^\infty(I))}
\lesssim \|\mu\|_{\ell^\beta}
\end{align}
holds for all orthonormal systems $(f_j)_j$ in $\dot H^s(\R)$ provided that
\begin{equation}\label{e:BKS}
s>
\begin{cases}
\max\big\{\frac14,\ \frac12\big(1-\frac1\beta\big)\big\}, & m>1,\\[4pt]
\max\big\{\frac12\big(1-\frac1\beta\big),\
\frac12-\frac m4-\frac{1-m}{2\beta}\big\}, & 0<m<1,
\end{cases}
\end{equation}
and this range is sharp up to endpoints.\footnote{Theorem 2.2 of \cite{BKS} is stated with the $L_x^1$ norm. The frequency-localized Schatten estimates in \cite[Section~6.2]{BKS} also imply the $L_x^\beta$ estimate \eqref{BKS-max} under the regularity condition \eqref{e:BKS}. For $m=2$, the argument is detailed in \cite[Appendix A]{JKS}.} Related maximal estimates were obtained for the wave and Boussinesq equations in \cite{KKS,KLS} and \cite{YLYL}, respectively, and for the Schr\"odinger equation on tori and waveguide manifolds by Bhimani--Choudhary \cite{BC26}.

The above results concern the vertical line $\gamma(x,t)=x$. For the Schr\"odinger equation $(m=2)$ in $\R^d$, $d\ge1$, the authors and Shiraki \cite{JKS} showed that, in the diagonal case $q=2\beta$, curves $\gamma$ satisfying \eqref{holder-k} with $\kappa=1$ and \eqref{lip-x} have the same regularity as the vertical line, up to endpoints. In particular, the analogue of \eqref{BKS-max} along such $\gamma$ holds for orthonormal systems in $H^s(\R^d)$ when
\[
s>\max\Big\{\frac d2\Big(1-\frac1\beta\Big),\
\frac{d}{2(d+1)}\Big(d-\frac{d-1}{\beta}\Big)\Big\},
\]
see \cite[Theorem 1.5 and Corollary 1.8]{JKS}. When $d=1$, this coincides with \eqref{e:BKS} for the case $m=2$.
The argument in \cite{JKS} relies on the Lipschitz regularity of $\gamma$ in $t$ and does not cover general curves that are only $\kappa$-H\"older in $t$ with $0<\kappa<1$. In this paper we establish orthonormal maximal estimates along $\gamma\in\Gamma(\kappa)$ for both $m>1$ and $0<m<1$. We state the case $m>1$ first.

\medskip

\textbf{The case $m>1$.}
For $0<\kappa\leq1$, $2\leq q\leq\infty$, and $1\leq\beta\leq\infty$, set
\begin{equation*}
s_{m,\kappa,q,\beta}:=
\max\Big\{ \frac14, \ \frac12-\frac{\min\{m\kappa,1\}}q, \ \frac 12-\frac{\min\{m\kappa,1\}}{2\beta} \Big\}.
\end{equation*}

\begin{thm}\label{thm:max}
Let $m>1$, $0<\kappa\leq1$, $2\le q \le \infty$, $1\leq\beta\leq\infty$, and $\gamma\in\Gamma(\kappa)$.
If $s>s_{m,\kappa,q,\beta}$, then
\begin{equation}\label{max-main}
 \Big\|\sum_j\mu_j|\Sf_j|^2\Big\|_
 {L_x^{\frac q2}(B;L_t^\infty(I))}
 \lesssim \|\mu\|_{\ell^\beta}
\end{equation}
holds for every orthonormal system $(f_j)$ in $H^s(\mathbb R)$ and every sequence $\mu=(\mu_j)\in \ell^\beta(\C)$.
Moreover, the regularity exponent is sharp up to endpoint.
\end{thm}

For the diagonal case $q=2\beta$, write $s_{m,\kappa,\beta}:=s_{m,\kappa,2\beta,\beta}$. The case $q=2$, $\beta=1$ recovers Theorem \ref{thm:single}\,$(i)$, while the case $q=\beta=\infty$ requires $s>1/2$ and follows from Bessel's inequality. The necessary conditions for the model curve are proved in Section \ref{sec:sharp}. Figure \ref{fig:th} illustrates the diagonal case $q=2\beta$.

\begin{figure}[htbp]
\centering
\definecolor{thcol}{RGB}{31,78,145}
\begin{tikzpicture}[
  x=2.55cm, y=3.4cm,
  every node/.style={font=\small},
  axisline/.style={->,>=stealth,gray!75!black,line width=0.45pt},
  curve/.style={thcol,line width=1.15pt,line cap=round,line join=round},
  ghost/.style={blue!45,line width=0.6pt,densely dashed},
  guide/.style={gray!55,densely dotted,line width=0.4pt},
  tick/.style={gray!75!black,line width=0.45pt},
  dot/.style={circle,fill=thcol,inner sep=0pt,minimum size=2.6pt},
  panellabel/.style={font=\footnotesize,anchor=north},
  region/.style={thcol!12},
]

\begin{scope}
  \def\slope{0.18}                      
  \pgfmathsetmacro{\yend}{0.5-\slope}
  \fill[region] (0,0) -- (0,{1/2}) -- (1,\yend) -- (1,0) -- cycle;
  \draw[guide] (0,{1/4}) -- (1.02,{1/4});
  \draw[guide] (1,0) -- (1,\yend);
  \draw[axisline] (-0.03,0) -- (1.14,0);
  \node[font=\footnotesize,above,inner sep=1.5pt] at (1.10,0.005) {$1/\beta$};
  \draw[axisline] (0,-0.03) -- (0,0.62);
  \node[font=\footnotesize,right,inner sep=2pt] at (0,0.60) {$s$};
  \draw[curve] (0,{1/2}) -- (1,\yend);
  \node[dot] at (0,{1/2}) {};
  \node[dot] at (1,\yend) {};
  \node[below left,font=\footnotesize,inner sep=1.5pt] at (0,0) {$0$};
  \draw[tick] (1,0) -- (1,-0.018) node[below,font=\footnotesize] {$1$};
  \draw[tick] (0,{1/4}) -- (-0.018,{1/4}) node[left,font=\footnotesize] {$\tfrac14$};
  \draw[tick] (0,{1/2}) -- (-0.018,{1/2}) node[left,font=\footnotesize] {$\tfrac12$};
  \node[thcol,font=\scriptsize,anchor=south west,inner sep=1pt]
    at (0.22,{0.5-0.22*\slope}) {$\frac12-\frac{m\kappa}{2\beta}$};
  \node[panellabel] at (0.5,-0.13) {(a)\; $0<m\kappa\le\frac12$};
\end{scope}

\begin{scope}[xshift=3.72cm]
  \def\slope{0.35}
  \pgfmathsetmacro{\xb}{1/(4*\slope)}
  \fill[region] (0,0) -- (0,{1/2}) -- (\xb,{1/4}) -- (1,{1/4}) -- (1,0) -- cycle;
  \draw[guide] (0,{1/4}) -- (1.02,{1/4});
  \draw[guide] (\xb,0) -- (\xb,{1/4});
  \draw[ghost] (\xb,{1/4}) -- (1,{0.5-\slope});

  \draw[axisline] (-0.03,0) -- (1.14,0);
  \node[font=\footnotesize,above,inner sep=1.5pt] at (1.10,0.005) {$1/\beta$};
  \draw[axisline] (0,-0.03) -- (0,0.62);
  \node[font=\footnotesize,right,inner sep=2pt] at (0,0.60) {$s$};
  \draw[curve] (0,{1/2}) -- (\xb,{1/4}) -- (1,{1/4});
  \node[dot] at (0,{1/2}) {};
  \node[dot] at (\xb,{1/4}) {};
  \node[below left,font=\footnotesize,inner sep=1.5pt] at (0,0) {$0$};
  \draw[tick] (\xb,0) -- (\xb,-0.018) node[below,font=\scriptsize] {$\frac1{2m\kappa}$};
  \draw[tick] (1,0) -- (1,-0.018) node[below,font=\footnotesize] {$1$};
  \draw[tick] (0,{1/4}) -- (-0.018,{1/4}) node[left,font=\footnotesize] {$\tfrac14$};
  \draw[tick] (0,{1/2}) -- (-0.018,{1/2}) node[left,font=\footnotesize] {$\tfrac12$};
  \node[thcol,font=\scriptsize,anchor=south west,inner sep=1pt]
    at (0.12,{0.5-0.12*\slope}) {$\frac12-\frac{m\kappa}{2\beta}$};
  \node[thcol,font=\scriptsize,above,inner sep=2pt] at (0.86,{1/4}) {$\frac14$};
  \node[panellabel] at (0.5,-0.13) {(b)\; $\frac12<m\kappa<1$};
\end{scope}

\begin{scope}[xshift=7.44cm]
  \def\slope{0.5}
  \fill[region] (0,0) -- (0,{1/2}) -- ({1/2},{1/4}) -- (1,{1/4}) -- (1,0) -- cycle;
  \draw[guide] (0,{1/4}) -- (1.02,{1/4});
  \draw[guide] ({1/2},0) -- ({1/2},{1/4});
  \draw[ghost] ({1/2},{1/4}) -- (1,0);
  \draw[axisline] (-0.03,0) -- (1.14,0);
  \node[font=\footnotesize,above,inner sep=1.5pt] at (1.10,0.005) {$1/\beta$};
  \draw[axisline] (0,-0.03) -- (0,0.62);
  \node[font=\footnotesize,right,inner sep=2pt] at (0,0.60) {$s$};
  \draw[curve] (0,{1/2}) -- ({1/2},{1/4}) -- (1,{1/4});
  \node[dot] at (0,{1/2}) {};
  \node[dot] at ({1/2},{1/4}) {};
  \node[below left,font=\footnotesize,inner sep=1.5pt] at (0,0) {$0$};
  \draw[tick] ({1/2},0) -- ({1/2},-0.018) node[below,font=\footnotesize] {$\tfrac12$};
  \draw[tick] (1,0) -- (1,-0.018) node[below,font=\footnotesize] {$1$};
  \draw[tick] (0,{1/4}) -- (-0.018,{1/4}) node[left,font=\footnotesize] {$\tfrac14$};
  \draw[tick] (0,{1/2}) -- (-0.018,{1/2}) node[left,font=\footnotesize] {$\tfrac12$};
  \node[thcol,font=\scriptsize,anchor=south west,inner sep=1pt]
    at (0.10,{0.5-0.10*\slope}) {$\frac12\big(1-\frac1\beta\big)$};
  \node[thcol,font=\scriptsize,above,inner sep=2pt] at (0.80,{1/4}) {$\frac14$};
  \node[panellabel] at (0.5,-0.13) {(c)\; $m\kappa\ge1$};
\end{scope}

\end{tikzpicture}
\caption{The critical regularity $s_{m,\kappa,\beta}$ in the diagonal case $q=2\beta$ for $m>1$.}
\label{fig:th}
\end{figure}
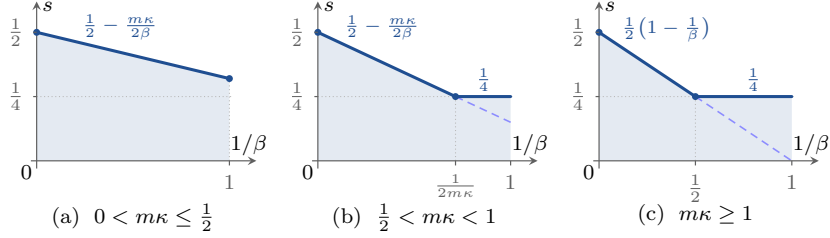

When $m>1$, $m\kappa\ge1$, and $q=2\beta$, the sharp regularity  coincides with that in \eqref{e:BKS}, so the curves in $\Gamma(\kappa)$ require no additional regularity compared with the vertical line. In particular, the non-tangential model curves $\gamma_\theta^\rho$ with $\rho\ge1$ belong to $\Gamma(1)$ and the maximal estimates along these curves have the same regularity as those along the vertical line.

\medskip

\textbf{The case $0<m<1$.} For $0<\kappa\leq1$, $2\le q\le \infty$, and $1\leq\beta\leq\infty$, define
\begin{equation*}
\sigma_{m,\kappa,q,\beta}
:=\max\Big\{ \frac12-\frac m4, \ \frac12-\frac{m\kappa}q, \ \frac12-\frac{m\kappa}{2\beta}\Big\}.
\end{equation*}

\begin{thm}\label{thm:small}
Let $0<m<1$, $0<\kappa\leq1$, $2\le q \le \infty$, $1\leq\beta\leq\infty$, and
$\gamma\in\Gamma(\kappa)$. If $s>\sigma_{m,\kappa,q,\beta}$, then
\begin{equation}\label{on-small}
 \Big\|\sum_j\mu_j\big|\Sf_j\big|^2\Big\|_
 {L_x^{\frac q2}(B;L_t^\infty(I))}
 \lesssim \|\mu\|_{\ell^\beta}
\end{equation}
for every orthonormal system $(f_j)$ in $H^s(\mathbb R)$ and every sequence $\mu=(\mu_j)\in\ell^\beta(\mathbb C)$.
Moreover, the regularity condition is sharp up to endpoint.
\end{thm}

For $q=2\beta$, write $\sigma_{m,\kappa,\beta}:=\sigma_{m,\kappa,2\beta,\beta}$. The case $q=2$, $\beta=1$ recovers Theorem \ref{thm:single}\,$(ii)$, while $q=\beta=\infty$ again requires $s>1/2$. Figure \ref{fig:concave} illustrates this diagonal case $q=2\beta$.

\begin{figure}[htbp]
\centering
\definecolor{thcol}{RGB}{31,78,145}
\begin{tikzpicture}[
  x=2.75cm, y=2.7cm,
  every node/.style={font=\small},
  axisline/.style={->,>=stealth,gray!75!black,line width=0.45pt},
  curve/.style={thcol,line width=1.15pt,line cap=round,line join=round},
  ghost/.style={blue!45,line width=0.6pt,densely dashed},
  guide/.style={gray!55,densely dotted,line width=0.4pt},
  tick/.style={gray!75!black,line width=0.45pt},
  dot/.style={circle,fill=thcol,inner sep=0pt,minimum size=2.6pt},
  panellabel/.style={font=\footnotesize,anchor=north},
  region/.style={thcol!12},
]

\begin{scope}
  \def\mval{0.6}
  \def\kap{0.35}
  \pgfmathsetmacro{\base}{0.5-\mval/4}
  \pgfmathsetmacro{\slope}{\mval*\kap/2}
  \pgfmathsetmacro{\yend}{0.5-\slope}
  \fill[region] (0,0) -- (0,{1/2}) -- (1,\yend) -- (1,0) -- cycle;
  \draw[ghost] (0,\base) -- (1.02,\base);
  \draw[axisline] (-0.03,0) -- (1.14,0);
  \node[font=\footnotesize,above,inner sep=1.5pt] at (1.10,0.005) {$1/\beta$};
  \draw[axisline] (0,-0.03) -- (0,0.62);
  \node[font=\footnotesize,right,inner sep=2pt] at (0,0.60) {$s$};
  \draw[curve] (0,{1/2}) -- (1,\yend);
  \node[dot] at (0,{1/2}) {};
  \node[dot] at (1,\yend) {};
  \node[below left,font=\footnotesize,inner sep=1.5pt] at (0,0) {$0$};
  \draw[tick] (1,0) -- (1,-0.018) node[below,font=\footnotesize] {$1$};
  \draw[tick] (0,{1/2}) -- (-0.018,{1/2}) node[left,font=\footnotesize] {$\tfrac12$};
  \draw[tick] (0,\base) -- (-0.018,\base)
    node[left,font=\scriptsize] {$\frac12-\frac m4$};

  \node[thcol,font=\scriptsize,anchor=south west,inner sep=1pt]
    at (0.18,{0.5-0.18*\slope}) {$\frac12-\frac{m\kappa}{2\beta}$};
  \node[panellabel] at (0.5,-0.13) {(a)\; $0<\kappa\le\frac12$};
\end{scope}

\begin{scope}[xshift=4.35cm]
  \def\mval{0.6}
  \def\kap{0.75}
  \pgfmathsetmacro{\base}{0.5-\mval/4}
  \pgfmathsetmacro{\slope}{\mval*\kap/2}
  \pgfmathsetmacro{\xb}{1/(2*\kap)}
  \pgfmathsetmacro{\yghost}{0.5-\slope}
  \fill[region] (0,0) -- (0,{1/2}) -- (\xb,\base) -- (1,\base) -- (1,0) -- cycle;
  \draw[guide] (0,\base) -- (1.02,\base);
  \draw[guide] (\xb,0) -- (\xb,\base);
  \draw[ghost] (\xb,\base) -- (1,\yghost);
  \draw[axisline] (-0.03,0) -- (1.14,0);
  \node[font=\footnotesize,above,inner sep=1.5pt] at (1.10,0.005) {$1/\beta$};
  \draw[axisline] (0,-0.03) -- (0,0.62);
  \node[font=\footnotesize,right,inner sep=2pt] at (0,0.60) {$s$};
  \draw[curve] (0,{1/2}) -- (\xb,\base) -- (1,\base);
  \node[dot] at (0,{1/2}) {};
  \node[dot] at (\xb,\base) {};
  \node[below left,font=\footnotesize,inner sep=1.5pt] at (0,0) {$0$};
  \draw[tick] (\xb,0) -- (\xb,-0.018)
    node[below,font=\scriptsize] {$\frac1{2\kappa}$};
  \draw[tick] (1,0) -- (1,-0.018) node[below,font=\footnotesize] {$1$};
  \draw[tick] (0,{1/2}) -- (-0.018,{1/2}) node[left,font=\footnotesize] {$\tfrac12$};
  \draw[tick] (0,\base) -- (-0.018,\base)
    node[left,font=\scriptsize] {$\frac12-\frac m4$};
  \node[thcol,font=\scriptsize,anchor=south west,inner sep=1pt]
    at (0.10,{0.5-0.10*\slope}) {$\frac12-\frac{m\kappa}{2\beta}$};
  \node[thcol,font=\scriptsize,above,inner sep=2pt]
    at (0.84,\base) {$\frac12-\frac m4$};
  \node[panellabel] at (0.5,-0.13) {(b)\; $\frac12<\kappa\le1$};
\end{scope}

\end{tikzpicture}
\caption{The critical regularity $\sigma_{m,\kappa,\beta}$ for $0<m<1$ in the diagonal case $q=2\beta$.}\label{fig:concave}
\end{figure}
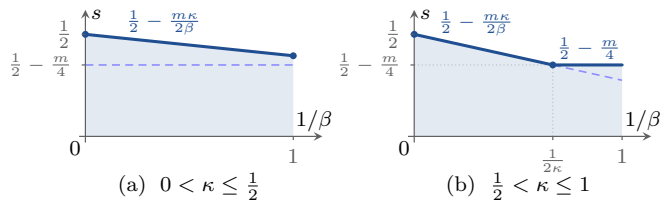

\begin{rmk}\label{rmk:mono}
Let $0<m<1$, $\rho\ge1$, $\theta>0$, $2\le q \le \infty$, and $1\le \beta \le \infty$. Since $\gamma_\theta^\rho\in\Gamma(1)$, Theorem \ref{thm:small} gives
\eqref{on-small} along $\gamma_\theta^\rho$ whenever
\begin{equation*}
 s>\sigma_{m,1,q,\beta}
 =\max\Big\{\frac12-\frac m4,\ \frac12-\frac mq, \ \frac12-\frac{m}{2\beta}\Big\}.
\end{equation*}
This condition is sharp up to the endpoint for each $\rho\ge1$ and
$\theta>0$. See Section \ref{sec:sharp}. For $q=2\beta<\infty$, the critical regularity for the vertical line $\gamma(x,t)=x$ in \eqref{e:BKS} is strictly smaller. Thus, when $0<m<1$, maximal estimates along these non-tangential model curves require more regularity than along the vertical line.
\end{rmk}

The sums in our maximal estimates can be viewed as densities of operators. For a finite orthonormal system $(f_j)$ in $H^s(\R)$, let $A_0$ be the operator with kernel $A_0(x,y)=\sum_j\mu_j f_j(x)\overline{f_j(y)}$, and set $A(t)=e^{it(-\Delta)^{m/2}}A_0\,e^{-it(-\Delta)^{m/2}}$. Then $A(t)$ solves the operator-valued equation
\[ i\partial_t A= [A,(-\Delta)^{m/2}],\quad A(0)=A_0.\]
The initial density is given by the finite sum $\rho_{A_0}(x)=\sum_j\mu_j|f_j(x)|^2$, and the density of $A(t)$ along the curve $\gamma$ is defined by
\begin{equation}\label{density definition}
 \rho_{A(t)}(\gamma(x,t))
 =\sum_j\mu_j\big|\Sf_j(x,t)\big|^2.
\end{equation}
Estimates for such densities play a role in the study of the Hartree equation for infinitely many particles \cite{LS14,LS15}. For infinite orthonormal systems under the hypotheses below, the densities are constructed by finite rank approximation in Sobolev type Schatten classes, as explained in Section \ref{subsec:C}.

Our maximal estimates imply the following pointwise convergence of the densities.

\begin{cor}\label{cor:density}
Let $0<\kappa\leq1$, $2\le q \le \infty$, $1\leq\beta<\infty$, and $\gamma\in\Gamma(\kappa)$. Assume either $m>1$ and
$s>s_{m,\kappa,q,\beta}$, or $0<m<1$ and $s>\sigma_{m,\kappa,q,\beta}$. Then, for every orthonormal system $(f_j)$ in $H^s(\mathbb R)$ and every $\mu\in\ell^\beta(\mathbb C)$,
\begin{equation*}
\lim_{t\to 0+}  \rho_{A(t)}(\gamma(x,t))=\rho_{A_0}(x),
 \quad\text{a.e. }x\in B.
\end{equation*}
\end{cor}

For infinite sums, Corollary \ref{cor:density} follows from Theorems \ref{thm:max} and \ref{thm:small} by the standard finite rank approximation. See \cite[Section~4.3]{BLN_Selecta} and \cite[Section~6.1]{BKS} for this approximation argument.

To prove the sufficiency for the maximal estimates, we first reduce to the diagonal case $q=2\beta$. By duality, it then suffices to establish the corresponding Schatten estimates. The kernel estimates for the frequency localized $TT^*$ operator, followed by interpolation between operator norm and Hilbert--Schmidt estimates yield the range $1\leq\beta\leq2$. Interpolation with Bessel's inequality gives the remaining range  $2<\beta\leq\infty$, completing the proof in the diagonal case.

For the necessity, single-function wave packets give the conditions involving $q$, together with the known lower bounds $s\geq1/4$ for $m>1$ and $s\geq1/2-m/4$ for $0<m<1$. The restriction involving $\beta$ is obtained by constructing orthonormal systems from spatially separated translates of the wave packets.

\subsection*{Organization of the paper}
The paper is organized as follows. In Section \ref{sec:pre}, we give the definition of the Schatten spaces and associated densities, and state the duality principle. We also reduce the maximal estimates to the diagonal case $q=2\beta$, and prove a Schatten bound for integral operators. Sections \ref{sec:convex} and \ref{sec:concave} are devoted to the proof of the sufficiency parts of Theorems \ref{thm:max} and \ref{thm:small}, respectively. In Section \ref{sec:sharp}, we prove the necessary regularity conditions and establish the sharpness up to endpoints.

\subsection*{Notation}  
Throughout the paper, $A\lesssim B$ means $A\le CB$ for some constant $C>0$ that is independent of frequency scales, functions, and coefficient sequences, unless otherwise stated. In addition, we write $A\sim B$ if $A\lesssim B$ and $B\lesssim A.$

\section{Preliminaries}\label{sec:pre}
In this section, we recall the Schatten spaces, make precise the definition of the densities described in the introduction, and state the duality principle. We then reduce the problem to the diagonal case $q=2\beta$ and collect some Schatten bounds for integral operators used in the proofs.

\subsection{Schatten spaces $\mathcal C^p$}\label{subsec:C}
For $1\le p<\infty$, the Schatten class $\mathcal C^p=\mathcal C^p(L^2(\R^d))$ consists of compact operators $A$ on $L^2(\R^d)$ whose singular values $(s_j(A))$ belong to $\ell^p$. Here, the singular values are the eigenvalues of $(A^*A)^{1/2}$. The norm is $\|A\|_{\mathcal C^p}=\|(s_j(A))\|_{\ell^p}$. We also set $\mathcal C^\infty=\mathcal B(L^2(\R^d))$, the space of bounded operators with the operator norm. The class $\mathcal C^2$ is the Hilbert--Schmidt class. In particular, an integral operator $A$ on $L^2(\R^d)$ with kernel $K\in L^2(\R^d\times \R^d)$ satisfies $\|A\|_{\mathcal C^2}=\|K\|_{L^2(\R^d\times \R^d)}$. 

For $1\le p \le\infty$ and $s\in\R$, we define the Sobolev-type Schatten class $\mathcal C^{p,s}$ as the space of bounded operators $A:H^{-s}(\R^d)\to H^s(\R^d)$ satisfying
\[
\|A\|_{\mathcal C^{p,s}}
:=\big\|\langle D\rangle^s A\langle D\rangle^s\big\|_{\mathcal C^p(L^2(\R^d))}
<\infty .
\]
We refer to \cite{Simon} for further properties of Schatten spaces.

Let $1\le\beta<\infty$, let $(f_j)$ be an orthonormal system in $H^s(\R)$, and let $\mu=(\mu_j)\in\ell^\beta$. Then the operator $A_0:H^{-s}(\R) \to H^s(\R)$ defined by
\[
 A_0h=\sum_j\mu_j\langle h,f_j\rangle f_j
\]
belongs to $\mathcal C^{\beta,s}$ with $\|A_0\|_{\mathcal C^{\beta,s}}=\|\mu\|_{\ell^\beta}$, since $(\langle D\rangle^sf_j)$ is orthonormal in $L^2(\R)$. We set $A(t)=U_m(t)A_0U_m(t)^*$ with $U_m(t)=e^{it(-\Delta)^{m/2}}$. For finite sums, the densities are given by \eqref{density definition}. In general, if $s$ is as in Theorem~\ref{thm:max} or Theorem~\ref{thm:small}, these theorems show that the partial sums of $\sum_j\mu_j|\Sf_j|^2$ converge in $L^{q/2}_x(B;L^\infty_t(I))$, and we define $\rho_{A(t)}(\gamma(x,t))$ as the limit. Similarly, $\rho_{A_0}$ is defined by the limit of $\sum_{j=1}^N \mu_j|f_j|^2$ in $L^{q/2}(B)$. Corollary~\ref{cor:density} then follows from this estimate by the standard approximation argument. See \cite[Section~4.3]{BLN_Selecta} and \cite[Section~6.1]{BKS}.

\subsection{Duality}

The following duality principle is due to Frank and Sabin \cite[Lemma 3]{FS_AJM} (see also \cite[Proposition 2.1]{BHLNS} for the mixed-norm formulation).

\begin{lem}\label{lem:dual}
Let $1\leq\beta\leq\infty$ and let
$T:L^2(\mathbb R)\to L_x^{2\beta}(B;L_t^\infty(I))$ be a bounded
linear operator.  The estimate
\begin{equation}\label{dual:beta}
 \Big\|\sum_j\mu_j|Tg_j|^2\Big\|_{L_x^\beta L_t^\infty}
 \leq A\|\mu\|_{\ell^\beta}
\end{equation}
for every orthonormal system $(g_j)\subset L^2(\mathbb R)$ and every $\mu=(\mu_j)\in\ell^\beta$ is equivalent, with the same optimal constant, to
\begin{equation*}
 \|WTT^*\overline W\|_{\mathcal C^{\beta'}}
 \leq A\|W\|_{L_x^{2\beta'}L_t^2}^2
\end{equation*}
for every $W\in L_x^{2\beta'}(B;L_t^2(I))$. Here $\beta'$ is the  H\"older conjugate of $\beta$.
\end{lem}

As an application of the duality (Lemma \ref{lem:dual}), we state the following lemma, whose proof can be found in \cite{JKS}. This lemma will be used to compose $T$ with frequency-localized multipliers.

\begin{lem}[{\cite[Lemma 2.3]{JKS}}]\label{lem:TB-T1}
Let $1\leq\beta\leq\infty$, let $T$ be as in Lemma~\ref{lem:dual},
and suppose that \eqref{dual:beta} holds.  If
$Q\in\mathcal B(L^2(\mathbb R))$, then
\begin{equation*}
 \Big\|\sum_j\mu_j|TQg_j|^2\Big\|_{L_x^\beta L_t^\infty}
 \leq A\|Q\|_{L^2\to L^2}^2\|\mu\|_{\ell^\beta}
\end{equation*}
for every orthonormal system $(g_j)\subset L^2(\mathbb R)$ and every $\mu=(\mu_j)\in\ell^\beta$.
\end{lem}

\subsection{Reduction to the diagonal case}\label{subsec:diag}
Even though the estimates \eqref{max-main} and \eqref{on-small} involve two parameters $q$ and $\beta$, it suffices to prove them when $q=2\beta$. We only consider \eqref{max-main}, since \eqref{on-small} is similar.  Note that 
\[
s_{m,\kappa,q,\beta}=\max\{s_{m,\kappa,q/2},s_{m,\kappa,\beta}\}.
\]
When $q\le2\beta$, H\"older's inequality gives
\[
\Big\|\sum_j\mu_j|\Sf_j|^2\Big\|_
{L_x^{\frac q2}(B;L_t^\infty(I))}
\lesssim \Big\|\sum_j\mu_j|\Sf_j|^2\Big\|_
{L_x^{\beta}(B;L_t^\infty(I))} \lesssim\|\mu\|_{\ell^\beta}
\]
for $s>s_{m,\kappa,\beta}=s_{m,\kappa,q,\beta}$. When $q\ge2\beta$, the diagonal estimate with $\beta$ replaced by $q/2$ and the embedding $\ell^\beta\subset\ell^{q/2}$ give
\[
\Big\|\sum_j\mu_j|\Sf_j|^2\Big\|_
 {L_x^{\frac q2}(B;L_t^\infty(I))}
 \lesssim\|\mu\|_{\ell^{\frac q2}}\le\|\mu\|_{\ell^\beta}
\]
for $s>s_{m,\kappa,q/2}=s_{m,\kappa,q,\beta}$. Therefore, we assume $q=2\beta$ in Sections~\ref{sec:convex} and \ref{sec:concave}.

\subsection{Schatten bound for integral operators}\label{sec:24}
The following Schatten bound is to be used in both convex $m>1$ and concave $0<m<1$ cases.
Let $\chi_*\in C_0^\infty(\R)$ be even with $0\le\chi_*\le1$, $\chi_*(z)=1$ for $|z|\le2$, nonincreasing as a function of $|z|$, and $\operatorname{supp}\chi_*\subset[-4,4]$. Set $\chi(z)=\chi_*(z)-\chi_*(2z)$ and $\chi_\ell(z)=\chi(2^\ell z)$ for $\ell\ge0$. Then $\chi_\ell\ge0$ and $\sum_{\ell\ge0}\chi_\ell(z)=1$ for $0<|z|<2$.

\begin{lem}\label{lem:DS}
Let $\ell\geq0$.  Suppose that
$\mathcal T_{\ell}$ is an integral operator on $L^2(B\times I)$ whose
kernel $K_\ell$ satisfies
\begin{equation*}
 |K_\ell(x,y,t,t')|
 \leq A_\ell \,\chi_\ell(x-y)
\end{equation*}
for some $A_\ell>0$.  Then, for $2\le q\le\infty$ and
$W_1,W_2\in L_x^{2q}(B;L_t^2(I))$,
\begin{equation}\label{e:Sdyadic}
\big\|W_1\mathcal T_\ell\overline{W_2}\big\|_{\mathcal C^q}
\lesssim
A_\ell2^{-\ell(1-\frac 1q)}
\|W_1\|_{L_x^{2q}(B;L_t^2(I))}
\|W_2\|_{L_x^{2q}(B;L_t^2(I))}.
\end{equation}
\end{lem}

\begin{proof}
We first prove the estimate when $q=\infty$ by duality.  For $F,G\in L^2(B\times I)$, the kernel bound and the
Cauchy--Schwarz inequality in the time variables give
\begin{align*}
|\langle W_1\mathcal T_\ell\overline{W_2}F,G\rangle|
&\lesssim A_\ell
\|W_1\|_{L_x^\infty L_t^2}
\|W_2\|_{L_x^\infty L_t^2}\\
&\quad\times
\iint_{B\times B}\chi_\ell(x-y)
\|F(y,\cdot)\|_{L_t^2(I)}
\|G(x,\cdot)\|_{L_t^2(I)}\,dx\,dy.
\end{align*}
Since $\|\chi_\ell\|_{L^1}\lesssim2^{-\ell}$, Young's
convolution inequality yields
\begin{equation*}
|\langle W_1\mathcal T_\ell\overline{W_2}F,G\rangle|
\lesssim A_\ell2^{-\ell}
\|W_1\|_{L_x^\infty L_t^2}
\|W_2\|_{L_x^\infty L_t^2}\|F\|_{L^2_{x,t}}\|G\|_{L^2_{x,t}}.
\end{equation*}
This and duality imply
\begin{equation*}
\|W_1\mathcal T_\ell\overline{W_2}\|_{L^2\to L^2}
\lesssim A_\ell2^{-\ell}
\|W_1\|_{L_x^\infty L_t^2}
\|W_2\|_{L_x^\infty L_t^2},
\end{equation*}
which is \eqref{e:Sdyadic} for $q=\infty$.

Now we prove \eqref{e:Sdyadic} when $q=2$. The upper bound for the kernel and  Young's convolution inequality give
\begin{align*}
\|W_1\mathcal T_\ell\overline{W_2}\|_{\mathcal C^2}^2
&\lesssim A_\ell^2
\iint_{B\times B}\chi_\ell^2(x-y)
\|W_1(x,\cdot)\|_{L_t^2(I)}^2
\|W_2(y,\cdot)\|_{L_t^2(I)}^2\,dx\,dy\\
&\lesssim A_\ell^2 2^{-\ell}
\|W_1\|_{L_x^4L_t^2}^2
\|W_2\|_{L_x^4L_t^2}^2.
\end{align*}
Consequently, we have
\begin{equation*}
\|W_1\mathcal T_\ell\overline{W_2}\|_{\mathcal C^2}
\lesssim A_\ell2^{-\ell/2}
\|W_1\|_{L_x^4L_t^2}
\|W_2\|_{L_x^4L_t^2}.
\end{equation*}
For $q\in (2,\infty)$, the bilinear complex interpolation between the $\mathcal C^\infty$ and $\mathcal C^2$ estimates proves the desired result. For the complex interpolation between Schatten spaces, we refer to \cite[Theorem~2.9 and Remark~1, p.~23]{Simon}.
\end{proof}

\section{Convex case: $m>1$}\label{sec:convex}
Throughout this section, we fix $m>1$, $0<\kappa\leq1$ and
$\gamma\in\Gamma(\kappa)$, and prove the sufficiency part of Theorem \ref{thm:max}. 

\subsection{Kernel estimates}
We choose real-valued functions $\psi_0\in C_0^\infty((-2,2))$ and
$\psi\in C_0^\infty((-2,-1/2)\cup(1/2,2))$ so that $\psi_0+\sum_{k\ge1}\psi(2^{-k}\cdot)=1$ on $\R$. For each dyadic $\lambda\geq2$, define the Littlewood--Paley projection $P_\lambda$ and the low-frequency projection $P_{0}$ by $ \widehat{P_\lambda f}=\psi(\lambda^{-1}\cdot)\widehat f$ and
 $\widehat{P_{0}f}=\psi_0\widehat f$, respectively. For each dyadic $\lambda\ge2$, set 
\begin{align}\label{Tlambda}
T_{\lambda}f(x,t):= S_m^\gamma P_\lambda f(x,t).
\end{align}
Let $K_\lambda$ be the kernel of $T_{\lambda}T_{\lambda}^*$. Then
\begin{align}\label{Ker-exp}
K_\lambda(x,y,t,t')=(2\pi)^{-1}\int_\mathbb R e^{i\Phi(x,y,t,t';\xi)}\psi(\xi/\lambda)^2\,d\xi
\end{align}
where the phase function is given by
\begin{align*}
\Phi(x,y,t,t';\xi)=\big(\gamma(x,t)-\gamma(y,t')\big)\xi
+(t-t')|\xi|^m.
\end{align*}

The next lemma is an analogue of Lemma 3.3 of Cho--Lee--Vargas~\cite{CLV}. For completeness, we include the proof.

\begin{lem}\label{lem:ker}
Let $m>1$, $0<\kappa\leq1$, and $\gamma\in\Gamma(\kappa)$. For every dyadic $\lambda\ge2$ and $(x,t),(y,t')\in B\times I$,
\begin{equation}\label{e:kernel1}
|K_\lambda(x,y,t,t')|
\lesssim 
\lambda(1+\lambda|x-y|)^{-1/2}
+\min\big\{\lambda,\lambda^{1-m/2}|x-y|^{-1/(2\kappa)}\big\}.
\end{equation}
\end{lem}
\begin{proof}
By the change of variables $\xi \rightarrow \lambda\xi$, we write
\begin{equation}\label{Kc}
 K_\lambda(x,y,t,t')
 =\frac{\lambda}{2\pi}\int_{\mathbb R}e^{i\Phi_\lambda(x,y,t,t';\xi)}\psi(\xi)^2\,d\xi,
\end{equation}
where $\Phi_\lambda(x,y,t,t';\xi)=\lambda\big(\gamma(x,t)-\gamma(y,t')\big)\xi +\lambda^m\big(t-t'\big)|\xi|^m$. Clearly, $|K_\lambda|\lesssim \lambda.$ To obtain  smaller bounds for $K_\lambda$, we observe that
\[\gamma(x,t)-\gamma(y,t')=\big(\gamma(x,t')-\gamma(y,t')\big)+\big(\gamma(x,t)-\gamma(x,t')\big),
\]
where the absolute value of the first term is comparable to $|x-y|$, and the second term is $O(|t-t'|^\kappa)$. We now  split into two cases $ |x-y|\ge C_0 |t-t'|^\kappa$ and $ |x-y|\le C_0 |t-t'|^\kappa$ with a sufficiently large constant $C_0>0$.  

We first assume that $ |x-y|\ge C_0|t-t'|^\kappa$ holds. Then $|\gamma(x,t)-\gamma(y,t')|\sim |x-y|$.  In addition, if
$|x-y|\ge C_1\lambda^{m-1}|t-t'|$ with a large constant $C_1>0$,  then
\[|\partial_\xi \Phi_\lambda|\ge \lambda |\gamma(x,t)-\gamma(y,t')|-
\lambda^m m|\xi|^{m-1}|t-t'|
 \gtrsim \lambda|x-y| \]
and  $|\partial_\xi^\alpha \Phi_\lambda|\lesssim \lambda|x-y|+\lambda^m|t-t'| \lesssim \lambda |x-y|$ for all $\alpha\ge1$ on $\operatorname{supp}  \psi$. Thus, repeated integration by parts gives
\begin{equation}\label{ibp}
|K_\lambda(x,y,t,t')| \lesssim \lambda (1+\lambda|x-y|)^{-N}    
\end{equation}
for any $N\ge1$. 
In the complementary region $|x-y|\le C_1\lambda^{m-1}|t-t'|$, a positive lower bound for $|\partial_\xi\Phi_\lambda|$ is no longer guaranteed. However, $|\partial^2_\xi\Phi_\lambda|\gtrsim \lambda^m|t-t'|\gtrsim \lambda|x-y|$ on $\operatorname{supp} \psi$. So,  van der Corput's lemma  gives 
\begin{equation*}
 |K_\lambda(x,y,t,t')|  \lesssim \lambda(1+\lambda|x-y|)^{-1/2},
\end{equation*}
which, together with \eqref{ibp}, gives the first term on the right-hand side of \eqref{e:kernel1}.

We next consider the case  $|x-y|\le C_0|t-t'|^\kappa$. Note that $|\partial_{\xi}^2 \Phi_\lambda|\sim \lambda^m|t-t'|\gtrsim \lambda^m|x-y|^{1/\kappa}$ on $\operatorname{supp}\psi$. Thus van der Corput's lemma gives
\begin{equation*}
 |K_\lambda(x,y,t,t')|
 \lesssim\lambda \big(1+\lambda^{m}|x-y|^{1/\kappa}\big)^{-1/2}.
\end{equation*}
This proves \eqref{e:kernel1}.
\end{proof}

\subsection{Spatial decomposition}\label{sub:dyadic}
Note that
\begin{equation}\label{2s-max}
2s_{m,\kappa,\beta}
=\max\Big\{
\frac12, \ 1-\frac1\beta, \ 1-\frac{m\kappa}{\beta}
\Big\}.
\end{equation}
Recall the cutoff functions $\chi_\ell$ defined in Section \ref{sec:24}. For $t,t'\in I$, define
\[
K_{\lambda,\ell}(x,y,t,t')
=
\chi_\ell(x-y)K_\lambda(x,y,t,t').
\]
By Lemma~\ref{lem:ker},
\begin{equation}\label{ker-2}
|K_{\lambda,\ell}(x,y,t,t')|
\lesssim \chi_\ell(x-y)M_{\lambda,\ell},
\end{equation}
where
\begin{equation}\label{ker-upper}
M_{\lambda,\ell}:=
\min\{\lambda,\lambda^{1/2}2^{\ell/2}\}
+\min\{\lambda,\lambda^{1-m/2}2^{\ell/(2\kappa)}\}.
\end{equation}
Combining this with Lemma \ref{lem:DS} gives the following.

\begin{lem}\label{lem:short}
Let $1\leq\beta\leq2$ and let $\lambda\geq2$ be dyadic.
For every $\varepsilon>0$ and
$W_1,W_2\in L_x^{2\beta'}(B;L_t^2(I))$,
\begin{equation*}
\|W_1T_{\lambda}T_{\lambda}^*\overline{W_2}\|_{
\mathcal C^{\beta'}}
\le C_\varepsilon
\lambda^{2s_{m,\kappa,\beta}+\varepsilon}
\|W_1\|_{L_x^{2\beta'}L_t^2}\|W_2\|_{L_x^{2\beta'}L_t^2}.
\end{equation*}
\end{lem}

\begin{proof}
Let $T_{\lambda,\ell}$ be the operator with kernel
$K_{\lambda,\ell}$.  Since $T_{\lambda}T_{\lambda}^*=\sum_{\ell\geq0} T_{\lambda,\ell}$, the triangle inequality in $\mathcal C^{\beta'}$ gives
\[
\|W_1T_{\lambda}T_{\lambda}^*\overline{W_2}\|_{\mathcal C^{\beta'}}
\le \sum_{\ell \ge0} \|W_1 T_{\lambda,\ell}\overline{W_2}\|_{\mathcal C^{\beta'}}.
\]
By \eqref{ker-2}, Lemma \ref{lem:DS} with $q=\beta'$ and $A_\ell=M_{\lambda,\ell}$ gives
\begin{equation*}
\|W_1 T_\lambda T_\lambda^* \overline{W_2}\|_{\mathcal C^{\beta'}}
\lesssim \Big(\sum_{\ell\geq0}M_{\lambda,\ell}2^{-\frac\ell\beta}\Big)
\|W_1\|_{L_x^{2\beta'}L_t^2}\|W_2\|_{L_x^{2\beta'}L_t^2}.
\end{equation*}
For the first term in \eqref{ker-upper}, we split the sum at $2^{-\ell}=\lambda^{-1}$. This gives
\begin{align*}
\sum_{\ell \ge0} 2^{-\frac \ell \beta}
\min\{\lambda,\lambda^{\frac12}2^{\frac \ell 2}\}
&\le \sum_{1\le 2^\ell\le \lambda} \lambda^{\frac12}2^{(\frac 12-\frac 1\beta)\ell}+ \sum_{2^\ell >\lambda} \lambda 2^{-\frac \ell\beta}
\lesssim
\lambda^{\max\{\frac 12,1- \frac1\beta\}} 
\end{align*}
for $1\le \beta<2$, and we have a logarithmic loss at $\beta=2$, which is absorbed by $C_\varepsilon \lambda^\varepsilon$.

 For the second term, we split the sum at $2^{-\ell}\sim \lambda^{-m\kappa}$, and therefore
\begin{align}
\sum_{\ell\ge0} 2^{-\frac \ell\beta}
\min\{\lambda,\lambda^{1-\frac m2}2^{\frac \ell{2\kappa}}\}
&\le \sum_{1\le 2^\ell\le \lambda^{m\kappa}} \lambda^{1-\frac m2}2^{(\frac 1{2\kappa}-\frac 1\beta)\ell}+ \sum_{2^\ell >\lambda^{m\kappa}} \lambda 2^{-\frac \ell\beta} \nonumber \\
&\leq C_\varepsilon
\lambda^{\max\{1-\frac m2, \ 1-\frac{m\kappa}\beta\}+\varepsilon}.\label{sum-ell}
\end{align}
Here, we have a logarithmic loss when $\beta=2\kappa$, which is absorbed by $C_\varepsilon \lambda^\varepsilon$. Since $m>1$, one has $1-m/2<1/2$.  The two estimates and \eqref{2s-max} prove the lemma.
\end{proof}

\subsection{The frequency-localized maximal estimate}
We now establish  frequency-localized maximal estimates in the orthonormal setting. 
\begin{prop}\label{prop:freq}
Let $1\leq\beta\leq\infty$.  For every $\varepsilon>0$ and every
dyadic $\lambda\geq2$,
\begin{equation}\label{freq-beta}
\Big\|\sum_j\mu_j| S_m^\gamma P_\lambda g_j|^2\Big\|_
{L_x^\beta(B;L_t^\infty(I))}
\leq C_\varepsilon
\lambda^{2s_{m,\kappa,\beta}+\varepsilon}
\|\mu\|_{\ell^\beta}
\end{equation}
for every orthonormal system $(g_j)$ in $L^2(\mathbb R)$ and every
$\mu=(\mu_j)\in\ell^\beta$.
\end{prop}

\begin{proof} 
We first consider the case $1\leq\beta\leq2$.  Taking $W_1=W_2$ in Lemma~\ref{lem:short} and applying Lemma \ref{lem:dual}, we obtain
\begin{equation*}
 \Big\|\sum_j\mu_j|S_m^\gamma P_\lambda g_j|^2\Big\|_
 {L_x^\beta(B;L_t^\infty(I))}
 \leq C_\varepsilon
 \lambda^{2s_{m,\kappa,\beta}+\varepsilon}
 \|\mu\|_{\ell^\beta},
\end{equation*}
which proves \eqref{freq-beta} for $1\leq\beta\leq2$. 

We next consider $2<\beta\leq\infty$.  For a fixed $(x,t)$, Bessel's inequality gives
\begin{equation*}
 \sum_j|S_m^\gamma P_\lambda g_j(x,t)|^2
 =(2\pi)^{-2} \sum_j \big|\big\langle
 e^{i\gamma(x,t)\xi+it|\xi|^m}\psi(\lambda^{-1}\xi),
 \overline{\widehat g_j}\big\rangle\big|^2
 \lesssim\|\psi(\lambda^{-1}\cdot)\|_{L^2}^2\lesssim \lambda.
\end{equation*}
Here, Plancherel's theorem shows that $(\frac1{\sqrt{2\pi}}\widehat{g}_j)$ is orthonormal in $L^2(\mathbb{R})$.
Consequently,
\begin{equation}\label{freq-inf}
 \Big\|\sum_j\mu_j|S_m^\gamma P_\lambda g_j|^2\Big\|_{L_{x,t}^\infty}
 \lesssim \lambda\|\mu\|_{\ell^\infty}.
\end{equation}
Since $2s_{m,\kappa,\infty}=1$, this proves \eqref{freq-beta} for $\beta=\infty$. Complex interpolation between $\ell^p\to L_x^p(B;L_t^\infty(I))$ for $p=2,\infty$ gives the exponent, up to $C_\varepsilon \lambda^\varepsilon$-loss,
\[
\frac2\beta\max\big\{\frac12,1-\frac{m\kappa}{2}\big\}+(1-\frac2\beta).
\]
If $m\kappa\geq1$, this equals $1-\frac 1\beta$, whereas if $m\kappa\leq1$, it equals $1-\frac{m\kappa}\beta$. In either case it is exactly $2s_{m,\kappa,\beta}$.  This proves \eqref{freq-beta} and completes the proof.
\end{proof}

\subsection{Proof of Theorem~\ref{thm:max}: the sufficiency part}\label{subsec:proof}
Let $1\le \beta \le \infty$ and $s>s_{m,\kappa,\beta}$, and let $(f_j)$ be an orthonormal family in $H^s(\R)$. Set $g_j=\langle D\rangle^sf_j$, so that $(g_j)$ is orthonormal in $L^2(\mathbb R)$.  Choose a multiplier $\widetilde P_\lambda$ with symbol supported in $\lambda/4\le|\xi| \le 4\lambda$ such that $\widetilde P_\lambda P_\lambda = P_\lambda$.  For $\lambda\geq2$, let $Q_{\lambda,s}:=\lambda^s\widetilde P_\lambda\langle D\rangle^{-s}$. This operator is uniformly bounded on $L^2$, while $P_\lambda\langle D\rangle^{-s} =\lambda^{-s}P_\lambda Q_{\lambda,s}$. Proposition~\ref{prop:freq} and Lemma~\ref{lem:TB-T1} give for $\lambda\ge2$
\begin{equation}\label{SPD}
 \Big\|\sum_j\mu_j
 |S_m^\gamma P_\lambda\langle D\rangle^{-s}g_j|^2\Big\|_
 {L_x^\beta L_t^\infty}
 \leq C_\varepsilon
 \lambda^{2s_{m,\kappa,\beta}+\varepsilon-2s}
 \|\mu\|_{\ell^\beta}.
\end{equation}

For the low-frequency part, the Bessel argument used in the proof of \eqref{freq-inf} yields, for every $1\le \beta\le \infty$, 
\begin{equation}\label{SPDlow}
\Big\|\sum_j\mu_j
|S_m^\gamma P_{0}\langle D\rangle^{-s}g_j|^2\Big\|_
{L_x^\beta L_t^\infty}
\lesssim \|\mu\|_{\ell^\beta}.
\end{equation}
By Minkowski's inequality for the sum over dyadic $\lambda$, together with
\eqref{SPDlow} and \eqref{SPD}, we have
\begin{align*}
\Big\|\sum_j\mu_j|S_m^\gamma f_j|^2\Big\|_{L_x^\beta L_t^\infty}^{1/2}
&\le\Big\|\sum_j|\mu_j||S_m^\gamma P_{0}f_j|^2\Big\|_{L_x^\beta L_t^\infty}^{1/2}
 +\sum_{\lambda\ge2}\Big\|\sum_j|\mu_j||S_m^\gamma P_\lambda f_j|^2\Big\|_{L_x^\beta L_t^\infty}^{1/2}\\
&\lesssim_\varepsilon\Big(1+\sum_{\lambda\ge2}
 \lambda^{s_{m,\kappa,\beta}+\frac\varepsilon2-s}\Big)\|\mu\|_{\ell^\beta}^{1/2}.
\end{align*}
Since $s>s_{m,\kappa,\beta}$, we may choose $\varepsilon>0$ so that $\varepsilon<2(s-s_{m,\kappa,\beta})$. Then the sum converges, which proves the claim.
\qed

\section{Concave case: $0<m<1$}\label{sec:concave}
Throughout this section, we fix $0<m<1$, $0<\kappa\leq1$ and $\gamma\in\Gamma(\kappa)$, and deal with the sufficiency in Theorem \ref{thm:small}.  As before, we postpone the proof of  necessity in Theorem \ref{thm:small} until Section \ref{sec:sharp}. We first establish kernel estimates.
\subsection{Kernel estimates for concave phases}

Let $P_\lambda$, $\lambda\geq2$, be the Littlewood--Paley projection operator defined as above and define $T_\lambda$ as in \eqref{Tlambda}. The kernel of $T_\lambda T_\lambda^*$ is $K_\lambda(x,y,t,t')$ defined by \eqref{Ker-exp}. The kernel estimate below follows from the argument in the proof of \cite[Lemma~4.1]{YZ}.

\begin{lem}\label{lem:ker2}
Let $0<m<1$, $0<\kappa\leq1$, and $\gamma\in\Gamma(\kappa)$.
For  $\lambda\geq2$ and 
$(x,t),(y,t')\in B\times I$, we have
\begin{equation}\label{k2}
|K_\lambda(x,y,t,t')|
\lesssim
\lambda\big(1+\lambda^m|x-y|^{\frac1\kappa}\big)^{-\frac12}.
\end{equation}
\end{lem}

\begin{proof}
We write $K_\lambda$ as in \eqref{Kc}, with the phase $\Phi_\lambda$. As before, we divide the proof into the two cases $|x-y|\ge C_0 |t-t'|^\kappa$ and $|x-y|\le C_0 |t-t'|^\kappa$ with $C_0$ sufficiently large.

Suppose first that $|x-y|\ge C_0 |t-t'|^\kappa$. In this case, the phase satisfies $|\partial_\xi\Phi_\lambda|\gtrsim\lambda|x-y|$. In fact, by \eqref{holder-k} and \eqref{lip-x},
\begin{equation}\label{low-gamma}
 |\gamma(x,t)-\gamma(y,t')| \sim |x-y|.
\end{equation} 
Since $0<m<1$, $\lambda \ge2$, $0<\kappa \le 1$, and $|t-t'|\le1$, we have
\begin{equation}\label{t-upper}
\lambda^m|t-t'|\le\lambda|t-t'|^\kappa
 \le C_0^{-1}\lambda|x-y|.
\end{equation}
Choosing $C_0$ sufficiently large,  \eqref{low-gamma} and \eqref{t-upper} imply
\[
|\partial_\xi \Phi_\lambda|\ge \lambda|\gamma(x,t)-\gamma(y,t')|-
m\lambda^m|t-t'||\xi|^{m-1}\gtrsim \lambda|x-y|
\]
on $\operatorname{supp}\psi$. Also, for every $\alpha \ge1$, we have $|\partial_\xi^\alpha\Phi_\lambda|\lesssim \lambda|x-y|+\lambda^m|t-t'|\lesssim\lambda|x-y|$ on $\operatorname{supp} \psi$ by \eqref{t-upper}. Thus, repeated integration by parts gives
\begin{equation}\label{ker-triv}
 |K_\lambda(x,y,t,t')|\lesssim_N\lambda(1+\lambda|x-y|)^{-N}
\end{equation}
for every $N\ge1$. Since $m<1\le1/\kappa$ and $|x-y|\le2$, we have $\lambda^m|x-y|^{1/\kappa}\lesssim \lambda|x-y|$. Thus, \eqref{ker-triv} implies  \eqref{k2}.

We now consider the case
$|x-y|\le C_0 |t-t'|^\kappa$.  Then
$|t-t'|\gtrsim|x-y|^{1/\kappa}$.  Since $|\xi|\sim1$ on
$\operatorname{supp}\psi$, one has
$|\partial_\xi^2 \Phi_\lambda|\gtrsim \lambda^m|t-t'| \gtrsim\lambda^m|x-y|^{1/\kappa}$
and van der Corput's lemma gives
\begin{align*}
|K_\lambda(x,y,t,t')|
&\lesssim\lambda (1+ \lambda^m |x-y|^{1/{\kappa}})^{-1/2}.
\end{align*}
This proves \eqref{k2}.
\end{proof}
In contrast to Lemma \ref{lem:ker}, the bound \eqref{k2} contains no term of the form $\lambda(1+\lambda|x-y|)^{-1/2}$.

\subsection{Bilinear Schatten estimates} 
Note that
\[
2\sigma_{m,\kappa,\beta}= \max \Big\{1-\frac m2,\, 1-\frac{m\kappa}\beta\Big\}.
\]
As in the convex case, a dyadic decomposition of the spatial variables in the kernel of  $T_\lambda T_\lambda^*$ and Lemma  \ref{lem:DS} imply the following bilinear estimates in Schatten spaces.
\begin{lem}\label{lem:sch}
Let $1\leq\beta\leq2$ and let $\lambda\geq2$.  For every $\varepsilon>0$ and $W_1,W_2\in L_x^{2\beta'}(B;L_t^2(I))$,
\begin{equation*}
\|W_1T_\lambda T_\lambda^*\overline{W_2}\|_{
\mathcal C^{\beta'}}
\le C_\varepsilon
\lambda^{2\sigma_{m,\kappa,\beta}+\varepsilon}
\|W_1\|_{L_x^{2\beta'}L_t^2}\|W_2\|_{L_x^{2\beta'}L_t^2}.
\end{equation*}
\end{lem}

\begin{proof}
We use the smooth partition of unity $\{ \chi_\ell\}$ defined in Section \ref{sec:24}. We decompose $T_\lambda T_\lambda^*=\sum_{\ell\geq0} T_{\lambda,\ell}$, where  $T_{\lambda,\ell}$ has kernel $K_{\lambda,\ell}:= \chi_\ell(x-y)K_\lambda(x,y,t,t')$. Note that $|x-y|\sim2^{-\ell}$ on $\operatorname{supp}\chi_\ell$, and Lemma \ref{lem:ker2} gives
\[
| K_{\lambda,\ell}(x,y,t,t')| \lesssim \chi_\ell(x-y)\widetilde M_{\lambda,\ell},
\]
where 
\begin{equation*}
\widetilde M_{\lambda,\ell}:=
\min\big\{\lambda, \lambda^{1-\frac m2}2^{\frac\ell{2\kappa}}\big\}.
\end{equation*}
Hence, the triangle inequality, followed by Lemma \ref{lem:DS} with $q=\beta'$, gives
\begin{equation*}
\|W_1T_\lambda T_\lambda^*\overline{W_2}\|_{
\mathcal C^{\beta'}}
\lesssim
\Big(\sum_{\ell\geq0}
\widetilde M_{\lambda,\ell}2^{-\frac\ell\beta}\Big)
\|W_1\|_{L_x^{2\beta'}L_t^2}\|W_2\|_{L_x^{2\beta'}L_t^2}.
\end{equation*}

The sum was already estimated in \eqref{sum-ell}, where the assumption $m>1$ was not used. At $\beta=2\kappa$, the logarithmic loss is absorbed by $\lambda^\varepsilon$. Hence it is bounded by $C_\varepsilon\lambda^{\max\{1-\frac m2,\,1-\frac{m\kappa}\beta\}+\varepsilon} =C_\varepsilon\lambda^{2\sigma_{m,\kappa,\beta}+\varepsilon}$, which proves the lemma.
\end{proof}

\subsection{The frequency-localized maximal estimate}
We now prove the analogue of Proposition~\ref{prop:freq} for $0<m<1$.
\begin{prop}\label{prop:freq2}
Let $1\leq\beta\leq\infty$.  For every $\varepsilon>0$ and $\lambda\geq2$, 
\begin{equation}\label{freq-c}
\Big\|\sum_j\mu_j|S_m^\gamma P_\lambda g_j|^2\Big\|_
{L_x^\beta(B;L_t^\infty(I))}
\lesssim_\varepsilon
\lambda^{2\sigma_{m,\kappa,\beta}+\varepsilon}
\|\mu\|_{\ell^\beta}
\end{equation}
holds for every orthonormal system $(g_j)$ in $L^2(\mathbb R)$,
and every $\mu=(\mu_j)\in\ell^\beta$.
\end{prop}
\begin{proof}
We argue as in the proof of Proposition \ref{prop:freq}. For $1\le \beta \le 2$, taking $W_1=W_2$ in Lemma \ref{lem:sch} and applying Lemma \ref{lem:dual} gives \eqref{freq-c}.

We next consider $2< \beta \le \infty$. The estimate \eqref{freq-inf} remains valid for $0<m<1$ and proves
\eqref{freq-c} when $\beta=\infty$, since $2\sigma_{m,\kappa,\infty}=1$.

Moreover, at $\beta=2$, \eqref{freq-c} holds with the exponent $2\sigma_{m,\kappa,2}+\varepsilon=1-m\kappa/2+\varepsilon$. We interpolate the associated linear density from $\mu \in \ell^p$ to $\sum_j\mu_j |S_m^\gamma P_\lambda g_j|^2 \in L_x^p(B;L_t^\infty(I))$ with $p=2,\infty$. The resulting exponent is
\[
\frac2\beta\Big(1-\frac{m\kappa}2\Big)+\Big(1-\frac2\beta\Big)
=1-\frac{m\kappa}\beta.
\]
Since $\kappa \le1$ and $\beta\ge2$, we have $1-\frac{m\kappa}\beta\geq1-\frac m2$. This exponent equals $2\sigma_{m,\kappa,\beta}$, so \eqref{freq-c} follows.
\end{proof}

\subsection{Proof of the sufficiency part of Theorem \ref{thm:small}} 
Let $s>\sigma_{m,\kappa,\beta}$, let $(f_j)$ be an orthonormal family in $H^s(\R)$, and set $g_j=\langle D\rangle^sf_j$. We repeat the argument in Section~\ref{subsec:proof}. Using Proposition~\ref{prop:freq2} and Lemma \ref{lem:TB-T1}, we obtain
\begin{equation*}
\Big\|\sum_j\mu_j
|S_m^\gamma P_\lambda\langle D\rangle^{-s}g_j|^2\Big\|_
{L_x^\beta L_t^\infty}
\leq C_\varepsilon
\lambda^{2\sigma_{m,\kappa,\beta}+\varepsilon-2s}
\|\mu\|_{\ell^\beta}.
\end{equation*}
The low-frequency estimate \eqref{SPDlow} also holds.
Choose $0<\varepsilon<2(s-\sigma_{m,\kappa,\beta})$. Summing over dyadic $\lambda$ as before, we obtain \eqref{on-small}.
\qed

\section{Sharpness of regularity}\label{sec:sharp}

We first use single function examples to obtain the necessary regularity for the $L_x^q$ maximal estimate. Orthonormal translates of these packets give the additional necessary conditions involving the exponent $\beta$. The same examples apply to both convex and concave phases.

\subsection{Necessary conditions of Theorems \ref{thm:max} and \ref{thm:small}}\label{subsec:nec}

\begin{prop}\label{prop:nece}
Let $m>0$, $m\neq1$, $0<\kappa\le1$, $2\le q \le \infty$, and $1\le\beta\le\infty$. Suppose that \eqref{max-main} (equivalently, \eqref{on-small}) holds for the curve $\gamma(x,t)=x-t^\kappa$, for all orthonormal systems in
$H^s(\mathbb R)$ and all coefficient sequences in $\ell^\beta$.
Then $s\geq s_{m,\kappa,q,\beta}$ if $m>1$, and $s\geq\sigma_{m,\kappa,q,\beta}$ if $0<m<1$.
\end{prop}

\textbf{(i) Proof of $s\ge \max\big\{\frac12-\frac{\min\{m\kappa,1\}}{q}, \frac12-\frac{\min\{m\kappa,1\}}{2\beta}\big\}$.}
Let $\psi\in C_0^\infty((-1,1))$ satisfy $\|\psi\|_2=1$ and $\widehat \psi \ge0$. We may fix $M\ge2$ such that $\int_{|\xi|>M}\widehat \psi(\xi) d\xi \le \frac 14\|\widehat \psi\|_1$.

For $R\gg1$, set $\mathcal J:=\{j \in \mathbb Z: 0< j\max\{R^{-1},R^{-m\kappa}\} \le 1/2\}$. For $j \in \mathcal J$, we define
\[
f_j (x) = c_0R^{\frac12}\langle D\rangle^{-s}\psi(10R(x-x_j))
\]
where $x_j=j\max\{R^{-1},R^{-m\kappa}\}$ and $c_0>0$ is chosen so that $\|f_j\|_{H^s}=1$.
Then $(f_j)_{j \in \mathcal J}$ forms an orthonormal system in $H^s(\R)$, since the supports of $\psi(10R(\cdot -x_j))$ are disjoint. Taking $\mu_j=1$ for $j\in \mathcal J$, and $\mu_j=0$ otherwise, we obtain $\|(\mu_j)\|_{\ell^\beta}\sim \min\{R^{\frac 1\beta}, R^{\frac{m \kappa}\beta}\}$.

Now we observe that
\[ |S_m^\gamma f_j(x,t)| \sim R^{\frac12}\Big| \int e^{i(10R(x-t^\kappa-x_j)\xi + t|10R\xi|^m)} \langle 10R\xi\rangle^{-s} \widehat \psi(\xi)\,d\xi \Big| \gtrsim R^{\frac12-s},
\]
provided $|t|\le cR^{-m}$ and $|x-t^\kappa-x_j| \le cR^{-1}$ for a sufficiently small $c>0$, depending on $M$.
We set
\begin{equation*}
 E_j=\bigcup_{0\leq t\leq cR^{-m}}
\Big( [t^\kappa-cR^{-1},t^\kappa+cR^{-1}]+x_j\Big).
\end{equation*}
For all sufficiently large $R\ge1$, each $E_j$ is contained in $(-1,1)$ with $|E_j|\gtrsim\max\{R^{-1},R^{-m\kappa}\}$.

Since the centers $x_j$ are separated by $\max\{R^{-1},R^{-m\kappa}\}$, these intervals are mutually disjoint. Therefore, $\big| \cup_{j\in \mathcal J}E_j| \sim \max\{R^{-1},R^{-m\kappa}\} \cdot \min\{R,R^{m\kappa}\} \sim 1$.

We now apply \eqref{max-main} with $\mu_j=1$ for $j \in \mathcal J$. We obtain
\[
 R^{1-2s}
 \lesssim\Big\|\sum_{j \in \mathcal J}
 |S_m^\gamma f_j|^2\Big\|_{L_x^{q/2}(B;L_t^\infty(I))}
 \lesssim |\mathcal J|^{\frac 1\beta}
 \sim R^{\frac{\min\{m\kappa,1\}}\beta}.
\]
Letting $R\to\infty$, we obtain
$s\ge\frac12-\frac{\min\{m\kappa,1\}}{2\beta}$.

On the other hand, applying \eqref{max-main} to $f_1$ alone, with $\mu_1=1$ and $\mu_j=0$ for $j\geq2$, gives
\begin{equation*}
1\gtrsim \Big\||S_{m}^{\gamma}f_1|^2\Big\|_
 {L_x^{q/2}(B;L_t^\infty(I))}
 \gtrsim R^{1-2s}|E_1|^{\frac 2q}
 \gtrsim R^{1-2s-2\frac{\min\{m\kappa,1\}}q}.
\end{equation*}
Letting $R\to\infty$, we obtain
$s\geq\frac12-\frac{\min\{m\kappa,1\}}{q}$.

\smallskip

\textbf{(ii) Proof of the remaining conditions.}
For a nonzero $f\in H^s(\R)$ satisfying $\|f\|_{H^s}=1$, take $\mu_1=1$, and $\mu_j=0$ for $j\ge2$. Applying
\eqref{max-main}, since $q\ge2$, we obtain
\begin{equation*}
 \|S_m^\gamma f\|_{L_x^2(B;L_t^\infty(I))}
 \lesssim\|S_m^\gamma f\|_{L_x^q(B;L_t^\infty(I))}
 \lesssim\|f\|_{H^s}.
\end{equation*}
Thus, \eqref{maxL2} holds. The results of Cho--Lee--Vargas \cite{CLV} and Cho--Shiraki \cite{CS21} imply $s\geq1/4$ when $m>1$, while the result of Yuan--Zhao \cite{YZ} implies
$s\geq1/2-m/4$ when $0<m<1$. Combining these conditions with \textbf{(i)}  proves the proposition, since $\min\{m\kappa,1\}=m\kappa$ when $0<m<1$.
\qed

\subsection{Sharpness of Remark \ref{rmk:mono}}
 Let $0<m<1$, $\rho\ge1$, and $\theta>0$. Set $\gamma(x,t)=x-\theta t^\rho$. The condition $s>\sigma_{m,1,q,\beta}$ in Remark \ref{rmk:mono} is sharp up to the endpoint. The necessity of $s\ge\frac12-\frac m4$ follows from the single-function result in \cite{CS22}.
Adapting the arguments in \cite{CS22} and  Subsection \ref{subsec:nec}, we obtain the remaining necessary conditions $s\ge \frac12-\frac m{2\beta}$ and $s\ge \frac12-\frac mq$.

Indeed, we take $t$ near $1/2$ and write $t=1/2+\tau$. Observe that $(1/2+\tau)^\rho-(1/2)^\rho\sim \tau$ since $\frac{d}{dt}(t^\rho)>0$ near $t=1/2$. Take the functions $f_j$ defined in \textbf{(i)}, now with $x_j=jR^{-m}$ and $\mathcal J=\{j \in \Z: 0<jR^{-m}\le 1/4\}$ and set
\begin{equation*}
\widehat{g_j}(\xi)=e^{-i(1/2)|\xi|^m+i\theta (1/2)^\rho\xi}\,\widehat f_j(\xi).
\end{equation*}
Clearly, $(g_j)$ remains orthonormal in $H^s(\R)$.
Note that
\begin{equation*}
|S_m^{\gamma}g_j(x,t)|
 =\frac{c_0}{2\pi} R^{1/2}\Big|\int
 e^{i10R(x-x_j-\theta(t^\rho-(1/2)^\rho))\xi+iR^m(t-1/2)|10\xi|^m}
\langle 10R\xi\rangle^{-s}\widehat \psi(\xi)\,d\xi\Big|
\end{equation*}
which is bounded below by a constant times $R^{\frac12-s}$
whenever $ |t-1/2|\leq cR^{-m}$ and $|x-x_j-\theta(t^\rho-(1/2)^\rho)|\leq cR^{-1}$
by the same argument as in $\mathbf{(i)}$.
Since $\theta(t^\rho-(1/2)^\rho)$ ranges over an interval of length comparable to
$R^{-m}$ when $|t-1/2|\le cR^{-m}$, the corresponding sets
\[
E_j=\bigcup_{|t-1/2|\le cR^{-m}} \Big( [\theta (t^\rho-(1/2)^\rho)-cR^{-1},~ \theta (t^\rho-(1/2)^\rho)+cR^{-1}]+x_j\Big)
\]
are intervals of length comparable to $R^{-m}$. Since $E_j$ have bounded overlap, we have $|\cup_{j\in \mathcal J}E_j|\sim 1$.  Taking $\mu_j=1$ for $j\in \mathcal J$ and $\mu_j=0$ otherwise, we obtain $\|\mu\|_{\ell^\beta}\sim R^{m/\beta}$. Applying \eqref{on-small} as in \textbf{(i)} gives $R^{1-2s}\lesssim  R^{m/\beta}$, and hence
$s\ge\frac12-\frac m{2\beta}$.

Applying \eqref{on-small} to $g_1$ alone, as in \textbf{(i)}, gives
$R^{1-2s}R^{-2m/q}\lesssim1$, and hence $s\ge\frac12-\frac mq$.
\qed

\section*{Acknowledgements} 
This work was supported by G-LAMP RS-2024-00443714 (E. Jeong and H. Ko); RS-2024-00339824 (E. Jeong); JBNU research funds for newly appointed professors in 2024 and RS-2026-25471967 (H. Ko). The authors thank Shobu Shiraki for helpful discussions and collaboration on related work.

\section*{AI Statement}
The authors used ChatGPT (OpenAI, GPT-6) to assist with language editing and reference checking. The original mathematical ideas, arguments, and proofs were developed by the authors.

\end{document}